\documentclass[11pt, a4paper, oneside, reqno]{amsart}
\usepackage{graphicx}
\usepackage{float}
\usepackage{pstricks}
\usepackage{amscd}
\usepackage{amsmath}
\usepackage{amsxtra}
\usepackage[hidelinks]{hyperref}
\usepackage[T1]{fontenc}
\usepackage[utf8]{inputenc}
\usepackage{lipsum}
\usepackage{tikz}
\usepackage{amssymb, amsthm}
\usepackage{url}
\usepackage{enumerate}
\usepackage{mathrsfs}
\usepackage{epsfig}
\usepackage[active]{srcltx}
\usepackage{enumitem}

\usetikzlibrary{arrows}
\usetikzlibrary{calc}
\theoremstyle{plain}
\newtheorem{theorem}{Theorem}[section]
\newtheorem{c-theorem}{Construction theorem}[section]

\newtheorem{lemma}[theorem]{Lemma}
\newtheorem{proposition}[theorem]{Proposition}
\newtheorem{corollary}[theorem]{Corollary}

\theoremstyle{definition}
\newtheorem{definition}[theorem]{Definition}

\theoremstyle{remark}
\newtheorem{remark}[theorem]{Remark}

\makeatletter
\@namedef{subjclassname@2020}{%
\textup{2020} Mathematics Subject Classification}
\makeatother

\newcommand{\ncom}{\newcommand}
\ncom{\bq}{\begin{equation}}
\ncom{\eq}{\end{equation}}
\ncom{\beqn}{\begin{eqnarray*}}
\ncom{\eeqn}{\end{eqnarray*}}
\ncom{\beq}{\begin{eqnarray}}
\ncom{\eeq}{\end{eqnarray}}
\ncom{\nno}{\nonumber}
\ncom{\rar}{\rightarrow}
\ncom{\Rar}{\Rightarrow}
\ncom{\noin}{\noindent}
\ncom{\bc}{\begin{centre}}
\ncom{\ec}{\end{centre}}
\ncom{\sz}{\scriptsize}
\ncom{\rf}{\ref}
\ncom{\sgm}{\sigma}
\ncom{\Sgm}{\Sigma}
\ncom{\dt}{\delta}
\ncom{\Dt}{Delta}
\ncom{\lmd}{\lambda}
\ncom{\Lmd}{\Lambda}
\ncom{\eps}{\epsilon}
\ncom{\pcc}{\stackrel{P}{>}}
\ncom{\dist}{{\rm\,dist}}
\ncom{\im}{{\rm Im\,}}
\ncom{\sgn}{{\rm sgn\,}}
\ncom{\ba}{\begin{array}}
\ncom{\ea}{\end{array}}
\ncom{\eop}{\hfill{{\rule{2.5mm}{2.5mm}}}}
\ncom{\eof}{\hfill{{\rule{1.5mm}{1.5mm}}}}
\ncom{\hone}{\mbox{\hspace{1em}}}
\ncom{\htwo}{\mbox{\hspace{2em}}}
\ncom{\hthree}{\mbox{\hspace{3em}}}
\ncom{\hfour}{\mbox{\hspace{4em}}}
\ncom{\hsev}{\mbox{\hspace{7em}}}
\ncom{\vone}{\vskip 2ex}
\ncom{\vtwo}{\vskip 4ex}
\ncom{\vonee}{\vskip 1.5ex}
\ncom{\vthree}{\vskip 6ex}
\ncom{\vfour}{\vspace*{8ex}}
\ncom{\norm}{\|\;\;\|}
\ncom{\integ}[4]{\int_{#1}^{#2}\,{#3}\,d{#4}}
\ncom{\inp}[2]{\langle{#1},\,{#2} \rangle}
\ncom{\Inp}[2]{\Langle{#1},\,{#2} \Langle}
\ncom{\vspan}[1]{{{\rm\,span}\#1 \}}}
\ncom{\dm}[1]{\displaystyle {#1}}

\begin{document}
\title[Dirichlet-type spaces]{Dirichlet-type spaces of the unit bidisc and toral completely hyperexpansive operators}

\author[Santu Bera]{Santu Bera}

\address{Department of Mathematics and Statistics\\
Indian Institute of Technology Kanpur, 208016, India}

\email{santu20@iitk.ac.in, santuddd12@gmail.com}

\thanks{The author is supported through the PMRF Scheme (2301352).}



\begin{abstract}
We discuss a notion, originally introduced by Aleman in one variable, of Dirichlet-type space $\mathcal D(\mu_1,\mu_2)$ on the unit bidisc $\mathbb D^2,$ with superharmonic weights related to finite positive Borel measures $\mu_1,\mu_2$ on $\overline{\mathbb D}.$ 
The multiplication operators $\mathscr M_{z_1}$ and $\mathscr M_{z_2}$ by the coordinate functions $z_1$ and $z_2,$ respectively, are bounded on $\mathcal D(\mu_1,\mu_2)$ and the set of polynomials is dense in $\mathcal D(\mu_1,\mu_2).$ We show that the commuting pair $\mathscr M_{z}=(\mathscr M_{z_1},\mathscr M_{z_2})$ is a cyclic analytic toral completely hyperexpansive $2$-tuple on $\mathcal D(\mu_1,\mu_2).$ Unlike the one variable case, not all cyclic analytic toral completely hyperexpansive  pairs arise as multiplication $2$-tuple $\mathscr M_z$ on these spaces. In particular, we establish that a cyclic analytic toral completely hyperexpansive operator $2$-tuple $T=(T_1,T_2)$ satisfying $I-T^*_1 T_1-T^*_2T_2+T^*_1T^*_2T_1T_2=0$ and having a cyclic vector $f_0$ is unitarily equivalent to $\mathscr{M}_z$ on $\mathcal{D}(\mu_1, \mu_2)$ for some finite positive Borel measures $\mu_1$ and $\mu_2$ on $\overline{\mathbb{D}}$ if and only if $\ker T^*$, spanned by $f_0$, is a wandering subspace for $T$.

\end{abstract}

\maketitle

\section{Introduction \& Preliminaries}

 Let $\mathbb D$ and $\mathbb T$ denote the open unit disc and the unit circle in the complex plane $\mathbb C$. For a non-empty subset $S$ of $\mathbb C,$ $M_+(S)$ denotes the set of finite positive Borel measures on $S.$ Let $\mathcal H$ denote a complex separable Hilbert space and $\mathcal{B(H)}$ be the $C^*$-algebra  of bounded linear operators on $\mathcal H.$ For an operator $S\in \mathcal B(\mathcal H),$ $S^*$ denotes the Hilbert space adjoint of $S.$ A pair  
$T=(T_1,T_2)$ is called a {\it commuting pair} on $\mathcal H$ if $T_1, T_2 \in \mathcal{B(H)} $ and $T_i T_j=T_j T_i$ for $1\leq i,j \leq 2.$ 
 Following \cite[p.~56]{EL2018}, we say that a commuting pair $T=(T_1,T_2)$ on $\mathcal H$ is {\it analytic} if 
 \beqn
\bigcap_{k=0}^{\infty} \sum_{\substack{\alpha_1, \alpha_2 \geq 0\\ \alpha_1 +\alpha_2=k}} T_1^{\alpha_1}T_2^{\alpha_2} \mathcal H=\{0\}.
\eeqn
A commuting pair $T=(T_1,T_2)$ is called {\it cyclic} with cyclic vector $f_0\in \mathcal H$ if the closed linear span of $\{ T_1^{\alpha_1}T_2^{\alpha_2}f_0,\alpha_1,\alpha_2\in \mathbb Z_+\}$ equals to $\mathcal H.$ Let $\alpha, \beta \in \mathbb Z_+^2,$ we say $\alpha \leq \beta$ if $\alpha_i \leq \beta_i$ for  $i=1,2.$ For $\alpha\leq \beta$  denote
$\binom{\alpha}{\beta}=\binom{\alpha_1}{\beta_1}\binom{\alpha_2}{\beta_2}.$ A commuting pair $T=(T_1,T_2)$ on $\mathcal H$ is said to be {\it toral completely hyperexpansive }  (refer to \cite[Definition~1]{AS1999}, cf. \cite{A1990, A1996}) if for each $\alpha\in \mathbb Z_+^2\setminus \{0\},$ 
 \begin{align}\label{beta-alpha}  
 \beta_\alpha(T):= \sum_{\substack{\beta\in  \mathbb Z_+^2 \\ 0\leq \beta \leq \alpha}} (-1)^{|\beta|} \binom{\alpha}{\beta} T^{*\beta}T^\beta \leq 0. 
\end{align}
Using binomial expansion it is easy to check that for each $\alpha \in \mathbb Z^2_+,$ 
\beq\label{beta-e-j}
\beta_{\alpha+\epsilon_j}(T)=\beta_\alpha(T)-T^*_j\beta_\alpha(T)T_j, \quad  j=1,2,
\eeq
 where $\epsilon_1=(1,0)$ and $\epsilon_2=(0,1).$ 
For a commuting pair $T=(T_1,T_2)$ on $\mathcal H$ we define the defect operator as
\beq\label{defect-operator}
\beta_{(1,1)}(T)=I-T^*_1 T_1-T^*_2T_2+T^*_1T^*_2T_1T_2.
\eeq
From the identity \eqref{beta-e-j} it is clear that whenever the defect operator of $T$ is zero i.e. $\beta_{(1,1)}(T)=0,$ $\beta_\alpha(T)=0$ for all $\alpha \in \mathbb Z^2_+$ with $\alpha_1\alpha_2\neq 0.$  It is evident that whenever two completely hyperexpansive operators $T_1$ and $T_2$  commute and have zero defect operator, the commuting $2$-tuple  $T=(T_1,T_2)$ is toral completely hyperexpansive.
Not every toral completely hyperexpansive $2$-tuple has zero defect operator. Indeed, if we take $d\nu$ to be Lebesgue area measure on $[0,1]^2$ in \cite[Eq.~H]{AS1999} then the defect operator is non zero. In particular,
\beqn
&&\|e_0\|^2-\|T_1e_0\|^2-\|T_2e_0\|^2+\|T_1T_2e_0\|^2\\
&=&1-(1+b_1+\frac{1}{2})-(1+b_2+\frac{1}{2})+(1+b_1+b_2+\frac{3}{4})=-\frac{1}{4}< 0.
\eeqn

Richter \cite{R1991} introduced the notion of {\it Dirichlet-type space}  on the unit disc $\mathbb D$ with harmonic weight and proved that these spaces are model spaces for cyclic analytic $2$-isometries (see \cite[Theorem~5.1]{R1991}). Later,  in \cite[Chapter~IV]{A1993} Aleman generalized this notion by considering superharmonic weights as follows:
Let $\mu$ be a finite positive Borel measure on $\overline{\mathbb D}.$ For a holomorphic function $f$ on $\mathbb D$ consider 
\beq\label{diri-int-one-variable}
\mathcal D_\mu(f)=\int_{\mathbb D}|f'(z)|^2U_\mu(z)dA(z),
\eeq 
where $dA$ denotes the normalized Lebesgue area measure on $\mathbb D$ and $U_\mu$ is the superharmonic function on $\mathbb D$ given by
\beqn
U_{\mu}(w)=\int_{\mathbb D}\log \left |\frac{1-\overline{\zeta}w}{w-\zeta}\right |^2\frac{d\mu(\zeta)}{1-|\zeta|^2}+\int_{\mathbb T}\frac{1-|w|^2}{|\zeta-w|^2}d\mu(\zeta),\; w\in \mathbb D.
\eeqn
Note that any positive superharmonic function on $\mathbb D$ is of this form (see \cite[Theorem ~4.5.1]{R1995}). The Dirichlet type space $\mathcal D(\mu)$  is the collection of holomorphic function $f$ on $\mathbb D$ such that $\mathcal D_\mu(f)<\infty.$ These spaces are subspaces of the Hardy space $H^2(\mathbb D).$ Concerning the following norm 
\beq \label{norm-d-mu}
\|f\|^2_{\mathcal D(\mu)}=\|f\|^2_{H^2(\mathbb D)}+\mathcal D_\mu(f), \quad f\in \mathcal D(\mu),
\eeq
$\mathcal D(\mu)$ is a Hilbert space and the multiplication operator $\mathscr M_z$ by the coordinate function $z$ is a cyclic analytic completely hyperexpansive (see \cite[Eq~D]{A1996} and \cite[Theorem~1.10(i), p~76]{A1993}. Moreover, any cyclic analytic completely hyperexpansive operator is unitarily equivalent to $\mathscr M_z$ on $\mathcal D(\mu)$ for some $\mu\in M_+(\overline{\mathbb  D})$ (see \cite[Theorem~2.5, p.~79]{A1993}). For further details on these spaces, please refer to \cite{A1993, EKKMR2016, GNS2018, S2025}.

In \cite{BCG2024} (see also \cite{BGV2024}) a notion of Dirichlet-type spaces on unit bidisc with harmonic weights has been introduced and observed that the multiplication tuple $\mathscr M_z=(\mathscr M_{z_1},\mathscr M_{z_2})$ is a toral $2$-isometry, i.e. $\beta_\alpha(\mathscr M_z)=0$ for each $\alpha \in \{(2,0),(0,2),(1,1)\}.$ In this present paper we generalize the notion of Dirichlet-type space introduced in \cite{BCG2024}, by replacing the harmonic weights with superharmonic weights. Motivated by \cite[Definition~1.8]{A1993}, \cite[Definition~1.1]{BCG2024} and \cite[Eq~3.1]{R1991} we define the following:
\begin{definition} \label{definition-diri}
For $\mu_1, \mu_2 \in M_+(\overline{\mathbb D})$ and a holomorphic function $f$ on the unit bidisc $\mathbb D^2,$ {\it the Dirichlet integral} $\mathcal D_{\mu_1, \mu_2}(f)$ of $f$ is given by
\begin{align} \label{dirichlet-int}
\mathcal D_{\mu_1, \mu_2}(f) &:=\sup_{0 < r < 1}\int_{0}^{2\pi} \int_{\mathbb D}|\partial_1 f(z_1, re^{i \theta})|^2 U_{\mu_1}(z_1) \,dA(z_1)\frac{d\theta}{2\pi} \notag \\
&+ \sup_{0 < r < 1} \int_{0}^{2\pi} \int_{\mathbb D}|\partial_2 f(re^{i \theta}, z_2)|^2 U_{\mu_2}(z_2) \,dA(z_2)\frac{d\theta}{2\pi}.
\end{align}
Consider the {\it Dirichlet-type space} $$\mathcal D(\mu_1,\mu_2):=\{f\in H^2(\mathbb D^2):\mathcal D_{\mu_1,\mu_2}(f)<\infty\},$$ 
where $H^2(\mathbb D^2)$ denotes the {\it Hardy space} on the unit bidisc $\mathbb D^2$ (see \cite{Ru1969}).
\end{definition}

It is clear from the definition that $\mathcal D_{\mu_1,\mu_2}(f)$ defines a seminorm on the space $\mathcal D(\mu_1,\mu_2).$ So we consider the following norm on $\mathcal D(\mu_1,\mu_2)$
\begin{equation}\label{norm-on-d-mu12}
\|f\|^2:=\|f\|^2_{H^2(\mathbb D^2)}+\mathcal D_{\mu_1,\mu_2}(f), \quad f\in \mathcal D(\mu_1,\mu_2).
\end{equation}
With this above norm $\mathcal D(\mu_1, \mu_2)$ is a reproducing kernel Hilbert space (see Lemma~\ref{D-mu-rkhs}). If we assume $\mu_j(\mathbb D)=0$ for $j=1,2,$ $\mathcal D(\mu_1,\mu_2)$ coincides with the notion of Dirichlet-type spaces appeared in \cite{BCG2024}.

\subsection{Statement of the main theorem}
Before stating the main theorem let us recall that a subspace $\mathcal W$ of $\mathcal H$ is said to be {\it wandering} (see \cite[Definition~1.5]{BCG2024}, cf. \cite[p.~103]{H1961}) for a commuting pair $T=(T_1,T_2)$ on $\mathcal H$ if for $\alpha_1,\alpha_2,\beta_1,\beta_2 \in \mathbb Z_+,$ 
\beqn
 T^{\alpha_1}_1 \mathcal W & \perp & T^{\beta_1}_1T^{\beta_2}_2\mathcal W, \text{ whenever } \beta_2 \neq 0, \\
  T^{\alpha_2}_2 \mathcal W & \perp & T^{\beta_1}_1T^{\beta_2}_2\mathcal W,  \text{ whenever }\beta_1 \neq 0.
\eeqn 
\begin{theorem} \label{model-theorem}
Let $T=(T_1,T_2)$ be a commuting pair on $\mathcal H.$ Then the following statements are equivalent:
\begin{enumerate}[label=$(\Alph*)$]
\item $T$ is a cyclic analytic toral completely hyperexpansive $2$-tuple such that $\beta_{(1,1)}(T)=0,$ and $T$ possesses a cyclic vector $f_0 \in \ker T^*,$ where $\ker T^*$ is a  wandering subspace of $T,$
\item  there exist $\mu_1,\mu_2\in M_+(\overline{\mathbb D})$ such that $T$ is unitarily equivalent to $\mathscr M_z$ on $\mathcal D(\mu_1,\mu_2).$
\end{enumerate}
\end{theorem}

\begin{remark} 
  From \cite[Theorem~1]{R1988} we know that any cyclic completely hyperexpansive operator on a complex separable Hilbert space has the wandering subspace property. But this fact fails in two-variable. For details one is refer to \cite[Remark~2.5]{BCG2024} (cf. \cite[Example~6.8]{BEKS2017}).
\end{remark}

Theorem~\ref{model-theorem} presents an analogue of \cite[Theorem~2.5]{A1993}, (cf. \cite[Theorem~2.4]{BCG2024}, \cite[Theorem~5.1]{R1991} and \cite[Theorem~5.1]{BGV2024}).
In Section~\ref{polynomial-density} we discuss the polynomial density, Gleason problem and boundedness of the multiplication $2$-tuple $\mathscr M_z=(\mathscr M_{z_1},\mathscr M_{z_2})$ on $\mathcal D(\mu_1,\mu_2).$ A proof of Theorem~\ref{model-theorem} is presented in Section~\ref{representaion-theorem} along with some of its consequences.

\section{Polynomial Density and Gleason Problem}\label{polynomial-density}
 Let $g \in H^2(\mathbb D)$ and $\zeta \in \overline{\mathbb D}.$ If $\zeta \in \mathbb D,$ recall that the local Dirichlet integral of $g$ at $\zeta$ (see \cite[p.~74]{A1993}, \cite[Theorem~2.1]{EKKMR2016}) is defined as 
\begin{equation}\label{local-diri-int}
D_\zeta(g)=\left \| \frac{g-g(\zeta)}{z-\zeta}\right \|^2_{H^2(\mathbb D)}= \int_0^{2\pi} \left | \frac{g(e^{it})-g(\zeta)}{e^{it}-\zeta}\right |^2 \frac{dt}{2\pi}.
\end{equation}
If $\zeta \in \mathbb T$ and $g(\zeta):= \lim_{r\rar 1^-}g(r\zeta)$ exists, we use the same formula \eqref{local-diri-int} to denote the local Dirichlet integral $D_\zeta(g)$ of $g$ at $\zeta.$  Otherwise, we set $D_\zeta(g)=\infty$ (see \cite[p.~356]{RS1991}). For general $\mu \in M_+(\overline{\mathbb D}),$  \cite[Theorem~1.9, p.~74]{A1993} (cf. \cite[Proposition~2.2]{RS1991}) gives
\beq \label{d-zeta-int}
\mathcal D_\mu(g)=\int_{\overline{\mathbb D}} \mathcal D_\zeta(g)d\mu(\zeta),\quad g\in \mathcal D(\mu).
\eeq
Here is a two-variable analog of the above equation.
\begin{proposition}
For $f\in \mathcal D(\mu_1, \mu_2),$
\begin{align*}
\mathcal D_{\mu_1, \mu_2}(f)&=\sup_{0 < r < 1}\int_{0}^{2\pi} \int_{\overline {\mathbb D}}D_{\zeta_1}(f(\cdot, re^{i\theta}))d\mu_1(\zeta_1)\frac{d\theta}{2\pi} \notag \\
&+ \sup_{0 < r < 1} \int_{0}^{2\pi}\int_{\overline {\mathbb D}}D_{\zeta_2}(f( re^{i\theta},\cdot))d\mu_2(\zeta_2) \frac{d\theta}{2\pi}.
\end{align*}
\end{proposition}
\begin{proof}
    Let $f\in \mathcal D(\mu_1, \mu_2)$ so it belongs to $H^2(\mathbb D^2).$ By \cite[Lemma~3.2]{BCG2024}, for each $r \in (0,1)$ and $\theta \in [0,2\pi]$ the slice functions  $f(\cdot, re^{i \theta})$ and $ f(re^{i\theta}, \cdot)$ belong to $H^2(\mathbb D).$ From \eqref{dirichlet-int} we get that for each $r\in(0,1)$ and almost every $\theta\in [0,2\pi],$
$$\int_{\mathbb D}|\partial_1 f(z_1, re^{i \theta})|^2 U_{\mu_1}(z_1) dA(z_1), \int_{\mathbb D}|\partial_2 f(re^{i \theta}, z_2)|^2 U_{\mu_2}(z_2) dA(z_2)< \infty. $$
 In other words for each $r \in (0,1),$
\beq \label{f-r-theta-in-d-mu}
&&\text{ there exists a measure zero subset $\Omega_r \subseteq [0,2\pi]$ such that the slices} \notag \\
 &&  f(\cdot,re^{i\theta})\in \mathcal D(\mu_1) \text{ and } f(re^{i\theta},\cdot)\in \mathcal D(\mu_2) \text{ for } \theta \in [0,2\pi]\setminus \Omega_r.
\eeq
Thus \eqref{dirichlet-int} becomes
\begin{align}\label{Dirichlet-int-new}
\mathcal D_{\mu_1, \mu_2}(f)=\sup_{0 < r < 1}\int_{0}^{2\pi} \mathcal D_{\mu_1}(f(\cdot, re^{i \theta}))\frac{d\theta}{2\pi} + \sup_{0 < r < 1} \int_{0}^{2\pi}\mathcal D_{\mu_2}(f(re^{i\theta}, \cdot)) \frac{d\theta}{2\pi}.
\end{align}
Combining \eqref{d-zeta-int} and \eqref{Dirichlet-int-new} yields the result.
\end{proof}
In view of \cite[Lemma~1.1]{BCG2024}, for any holomorphic function $f$ on $\mathbb D^2$ and $\nu \in M_+(\mathbb D),$ the function $r\rightarrow \int_{\mathbb T}\int_{\mathbb D}|f(z,re^{i\theta})|^2d\nu(z) d\theta$ is increasing. So we can replace $\sup_{0<r<1}$ in \eqref{dirichlet-int} by $\lim_{r\to 1^-}.$ Thus for each $f\in \mathcal D(\mu_1, \mu_2),$ $\mathcal D_{\mu_1, \mu_2}(f)$ breaks into two parts as $\mathcal D_{\mu_1, \mu_2}(f)=I_{\mu_1, \mu_2}(f)+B_{\mu_1, \mu_2}(f),$ where $I_{\mu_1, \mu_2}(f)$ and $B_{\mu_1, \mu_2}(f)$ are the integrals correspond to $\mu_j|_{\mathbb D}$ and $\mu_j|_{\mathbb T},$ respectively, for $j=1,2$ and given by
\begin{align}\label{I-mu-formula}
I_{\mu_1, \mu_2}(f)&=\sup_{0 < r < 1}\int_{0}^{2\pi} \int_{\mathbb D}D_{\zeta_1}(f(\cdot, re^{i\theta}))d\mu_1(\zeta_1)\frac{d\theta}{2\pi} \notag \\
&+ \sup_{0 < r < 1} \int_{0}^{2\pi}\int_{\mathbb D}D_{\zeta_2}(f( re^{i\theta},\cdot))d\mu_2(\zeta_2) \frac{d\theta}{2\pi},
\end{align}
\begin{align*}
B_{\mu_1, \mu_2}(f)&=\sup_{0 < r < 1}\int_{0}^{2\pi} \int_{\mathbb T}D_{\zeta_1}(f(\cdot, re^{i\theta}))d\mu_1(\zeta_1)\frac{d\theta}{2\pi} \notag \\
&+ \sup_{0 < r < 1} \int_{0}^{2\pi}\int_{\mathbb T}D_{\zeta_2}(f( re^{i\theta},\cdot))d\mu_2(\zeta_2) \frac{d\theta}{2\pi}.
\end{align*}
The following lemmas are fundamental to prove the polynomial density and boundedness of the multiplication $2$-tuple $\mathscr M_z$ on $\mathcal D(\mu_1,\mu_2).$
\begin{lemma}\label{D-mu-rkhs}
The Dirichlet-type space $\mathcal D(\mu_1,\mu_2)$ is a reproducing kernel Hilbert space. If $\kappa:\mathbb D\times \mathbb D\to \mathbb C$ is the reproducing kernel of $\mathcal D(\mu_1,\mu_2),$ then for any $r\in (0,1),$ $\bigvee\{\kappa(\cdot,w):|w|<r\}=\mathcal D(\mu_1,\mu_2)$ and $\kappa(\cdot,0)=1.$
\end{lemma}
\begin{proof}
By replacing $P_{\mu_1}$ by $U_{\mu_1}$ and $P_{\mu_2}$ by $U_{\mu_2}$ in \cite[Lemma~3.1]{BCG2024} and arguing  similarly we get required result.  
\end{proof}

Let $\mu \in M_+(\overline{\mathbb D})$ and $g\in \mathcal D(\mu).$ For each $r \in (0,1)$ define the function $g_r$ on $\mathbb D$ as  $g_r(w):=g(rw),$ $w \in \mathbb D.$  Combining  \cite[Lemma~4.1, p.~87]{A1993} and \cite[Theorem~1.9, p.~74]{A1993} gives $D_\mu(g_r)\leq \frac{5}{2}D_\mu(g).$ Later, in \cite[Theorem~4.2]{EKKMR2016} this inequality is improved to
\begin{equation}\label{d-zeta-g-r}
D_\mu(g_r)\leq D_\mu(g).
\end{equation}
The following lemma provides a similar estimate as of \eqref{d-zeta-g-r} for $\mathcal D(\mu_1,\mu_2).$
For $R=(R_1,R_2)\in (0,1)^2$ and $f\in \mathcal O(\mathbb D^2)$ let $f_R(z)=f(R_1z_1,R_2z_2)$ for $z=(z_1,z_2)\in \mathbb D^2.$
\begin{lemma}
For any $R=(R_1,R_2)\in (0,1)^2$ and $f\in \mathcal D(\mu_1,\mu_2),$ 
\begin{equation*}
\mathcal D_{\mu_1,\mu_2}(f_R)\leq \mathcal D_{\mu_1,\mu_2}(f).
\end{equation*}
\end{lemma}
\begin{proof}
Let $f \in \mathcal D(\mu_1, \mu_2).$ Fix $R=(R_1,R_2)\in (0,1)^2.$ By \eqref{Dirichlet-int-new},
\begin{eqnarray*}
&&\mathcal D_{\mu_1, \mu_2}(f_R)\\
&=&\sup_{0 < r < 1}\int_{0}^{2\pi} \mathcal D_{\mu_1}(f_R(\cdot, re^{i\theta}))\frac{d\theta}{2\pi}
+ \sup_{0 < r < 1} \int_{0}^{2\pi}\mathcal D_{\mu_2}(f_R( re^{i\theta},\cdot)) \frac{d\theta}{2\pi}\\
&\overset{\eqref{d-zeta-g-r}} \leq & \sup_{0 < r < 1}\int_{0}^{2\pi} \mathcal D_{\mu_1}(f(\cdot, R_2re^{i \theta}))\frac{d\theta}{2\pi} +\sup_{0 < r < 1} \int_{0}^{2\pi}\mathcal D_{\mu_2}(f(R_1re^{i\theta}, \cdot)) \frac{d\theta}{2\pi}.\\
\end{eqnarray*}
Finally applying \cite[Lemma~1.1]{BCG2024} yields the result.

\end{proof}
The next lemma is a prototype of \cite[Lemma~3.7]{BCG2024} (cf. \cite[Theorem~ 7.3.1]{EKMR2014}) and the proof is similar so left to the reader. 
\begin{lemma}\label{f_R-f-estimate}
Let $f\in \mathcal D(\mu_1,\mu_2).$ Then 
$$\lim_{R_1,R_2\to 1^-}\mathcal D_{\mu_1,\mu_2}(f-f_R)=0.$$
\end{lemma}

As an application of Lemma~\ref{f_R-f-estimate} we show that the set of polynomial is dense in $\mathcal D(\mu_1,\mu_2).$
\begin{lemma}\label{polynomials-are-dense}
Polynomials are dense in $\mathcal D(\mu_1,\mu_2).$
\end{lemma}
\begin{proof}
 Let $f\in \mathcal D(\mu_1,\mu_2)$ and choose $\varepsilon>0.$ It is enough to show that there exists a polynomial $p$ such that $\|f-p\|_{\mathcal D(\mu_1,\mu_2)}<\varepsilon.$ By Lemma~\ref{f_R-f-estimate} there exists $R=(R_1,R_2)\in (0,1)^2$  such that $\|f-f_R\|_{\mathcal D(\mu_1,\mu_2)}<\varepsilon/2.$ Since $f_R$ is holomorphic in a neighborhood of $\overline{\mathbb D}^2,$ there exists a polynomial $p$ such that 
$$\|f_R-p\|_{\infty,\overline{\mathbb D}^2},\left\|\frac{\partial f_R}{\partial z_j}-\frac{\partial p}{\partial z_j}\right\|_{\infty,\overline{\mathbb D}^2}<\frac{\sqrt{\varepsilon}}{4\sqrt{M}},\quad j=1,2,$$
where  $M=\hbox{max }\{\int_{\mathbb D}U_{\mu_j}(w)dA(w):j=1,2\}+1.$ This together with the fact that the norm on $H^2(\mathbb D^2)$ is dominated by the norm $\|\cdot\|_{\infty,\overline{\mathbb D}^2}$ shows that $\|f_R-p\|_{\mathcal D(\mu_1,\mu_2)}<\varepsilon/2.$ Thus using triangle inequality we get that $\|f-p\|_{\mathcal D(\mu_1,\mu_2)}<\varepsilon.$ Hence the proof.
\end{proof}

This next lemma is very crucial to prove the boundedness of the multiplication tuple $\mathscr M_z=(\mathscr M_{z_1},\mathscr M_{z_2}).$
\begin{lemma}\label{I-zf-estimate}
    Let $f\in H^2(\mathbb D^2).$ Then
    \beqn
     I_{\mu_1,\mu_2}(z_1f)&=&\sup_{0<r<1}\int_{0}^{2\pi}\int_{\mathbb D}|f(\zeta_1,re^{i\theta})|^2d\mu_1(\zeta_1)\frac{d\theta}{2\pi} + I_{\mu_1,\mu_2}(f),\\
     I_{\mu_1,\mu_2}(z_2f)&=&\sup_{0<r<1}\int_{0}^{2\pi}\int_{\mathbb D}|f(re^{i\theta},\zeta_2)|^2d\mu_2(\zeta_2)\frac{d\theta}{2\pi} + I_{\mu_1,\mu_2}(f).
    \eeqn
\end{lemma}
\begin{proof}
     For each $r\in(0,1)$ and $\theta\in [0,2\pi]$  define $f_{r,\theta}(w):=f(w,re^{i\theta}),$ $w\in \mathbb D,$ then $f_{r,\theta}\in H^2(\mathbb D)$ (see \cite[Lemma~3.2]{BCG2024}). Since $H^2(\mathbb D)$ is closed under the multiplication of the coordinate function $w$ so $wf_{r,\theta} \in H^2(\mathbb D).$ Fixing $\zeta \in\mathbb D$ we know that for each $g\in H^2(\mathbb D),$ $\frac{g-g(\zeta)}{w-\zeta} \in H^2(\mathbb D).$ In particular, $g=wf_{r,\theta}$ gives $\frac{wf_{r,\theta}-(wf_{r,\theta})(\zeta)}{w-\zeta} \in H^2(\mathbb D)$ and 
    \begin{align*}
     \left \| \frac{wf_{r,\theta}-(wf_{r,\theta})(\zeta)}{w-\zeta}\right \|^2_{H^2(\mathbb D)}=\left \| f_{r,\theta}(\zeta)+w\frac{f_{r,\theta}-f_{r,\theta}(\zeta)}{w-\zeta}\right \|^2_{H^2(\mathbb D)}.
    \end{align*}
    As the constant functions are orthogonal to $wH^2(\mathbb D)$ in $H^2(\mathbb D)$ so the above equation becomes 
    \beq\label{wf-r}
     \left \| \frac{wf_{r,\theta}-(wf_{r,\theta})(\zeta)}{w-\zeta}\right \|^2_{H^2}&=&|f_{r,\theta}(\zeta)|^2+\left \|\frac{f_{r,\theta}-f_{r,\theta}(\zeta)}{w-\zeta}\right \|^2_{H^2}\\
     &=&|f(\zeta,re^{i\theta})|^2+\left \|\frac{f(\cdot,re^{i\theta})-f(\zeta,re^{i\theta})}{w-\zeta}\right \|^2_{H^2}. \notag
    \eeq
    Now \eqref{I-mu-formula} together with \eqref{local-diri-int} implies
   \begin{align}\label{I-mu-zf}
    &I_{\mu_1,\mu_2}(z_1f) \notag\\
    &=\sup_{0<r<1}\int_{0}^{2\pi}\int_{\mathbb D}\Big \|\frac{(z_1f)(\cdot,re^{i\theta})-(z_1f)(\zeta_1,re^{i\theta})}{z_1-\zeta_1}\Big \|^2_{H^2(\mathbb D)}d\mu_1(\zeta_1)\frac{d\theta}{2\pi} \notag \\
    &+ \sup_{0<r<1}\int_{0}^{2\pi}\int_{\mathbb D} \Big\|\frac{(z_1f)(re^{i\theta},\cdot)-(z_1f)(re^{i\theta},\zeta_2)}{z_2-\zeta_2}\Big\|^2_{H^2(\mathbb D)}d\mu_2(\zeta_2)\frac{d\theta}{2\pi}  \notag\\
    &\overset{\eqref{wf-r}}= \sup_{0<r<1}\int_{0}^{2\pi}\int_{\mathbb D}|f(\zeta_1,re^{i\theta})|^2d\mu_1(\zeta_1)\frac{d\theta}{2\pi}  \notag \\
    & + \sup_{0<r<1}\int_{0}^{2\pi}\int_{\mathbb D}\Big\|\frac{f(\cdot,re^{i\theta})-f(\zeta_1,re^{i\theta})}{z_1-\zeta_1}\Big \|^2_{H^2(\mathbb D)}d\mu_1(\zeta_1)\frac{d\theta}{2\pi} \notag \\
    & +\sup_{0<r<1}\int_{0}^{2\pi}\int_{\mathbb D} \Big \|\frac{f(re^{i\theta},\cdot)-f(re^{i\theta},\zeta_2)}{z_2-\zeta_2}\Big\|^2_{H^2(\mathbb D)}d\mu_2(\zeta_2)\frac{d\theta}{2\pi} \notag \\
    &\overset{\eqref{I-mu-formula}}=\sup_{0<r<1}\int_{0}^{2\pi}\int_{\mathbb D}|f(\zeta_1,re^{i\theta})|^2d\mu_1(\zeta_1)\frac{d\theta}{2\pi} + I_{\mu_1,\mu_2}(f).
   \end{align}
   Similarly on can derive the expression involving $I_{\mu_1,\mu_2}(z_2f).$ 
\end{proof}

Here we show that the coordinate functions $z_1$ and $z_2$ are multipliers of $\mathcal D(\mu_1,\mu_2).$
\begin{lemma}\label{boundedness of M_z}
Let $\mu_1,\mu_2 \in M_+(\overline{\mathbb D}).$ Then $\mathscr M_{z_1}$ and $\mathscr M_{z_2}$ are bounded linear operators on $\mathcal D(\mu_1,\mu_2).$
\end{lemma}
\begin{proof}
 Let $f\in \mathcal D(\mu_1,\mu_2).$ So both $I_{\mu_1,\mu_2}(f)$ and  $B_{\mu_1,\mu_2}(f)$ are finite.  From the proof of \cite[Lemma~3.4]{BCG2024}, there exists constant $C \geq 1$ such that
\beqn
B_{\mu_1,\mu_2}(z_1f) \leq C \Big ( \|f\|^2_{ H^2(\mathbb D^2)}+ B_{\mu_1,\mu_2}(f) \Big) < \infty.
\eeqn
Now we show that $I_{\mu_1,\mu_2}(z_1f)<\infty.$ In view of Lemma~\ref{I-zf-estimate} it is enough to show that 
\beq\label{require-int}
\sup_{0<r<1}\int_{0}^{2\pi}\int_{\mathbb D}|f(\zeta_1,re^{i\theta})|^2d\mu_1(\zeta_1)\frac{d\theta}{2\pi}<\infty.
\eeq
The following idea is motivated from \cite[Proposition~1.6]{A1993}. 
Let $h\in H^2(\mathbb D^2)$ and assume that $h(0,z_2)=0$ for all $z_2\in \mathbb D.$ By the division property of $H^2(\mathbb D^2)$ (see \cite[Remark~4.2]{BCG2024}) there exists $g\in H^2(\mathbb D^2)$ such that $h(z_1,z_2)=z_1g(z_1,z_2)$ for all $(z_1,z_2)\in \mathbb D^2.$ Then for each $r\in (0,1),$ $\theta \in [0,2\pi]$ and $\zeta_1\in \mathbb D,$ using \eqref{wf-r} we get
\begin{align}\label{h-zeta-estimate}
\left \| \frac{h(\cdot,re^{i\theta})-h(\zeta_1,re^{i\theta})}{z_1-\zeta_1}\right \|^2_{H^2}
&= |g(\zeta_1,re^{i\theta})|^2+\left \|\frac{g(\cdot,re^{i\theta})-g(\zeta_1,re^{i\theta})}{z_1-\zeta_1}\right \|^2_{H^2} \notag \\
&\geq  |g(\zeta_1,re^{i\theta})|^2  \notag \\
&\geq  |\zeta_1|^2|g(\zeta_1,re^{i\theta})|^2 \notag \\
&= |h(\zeta_1,re^{i\theta})|^2.
\end{align}
Thus, 
\begin{align} \label{f-in-L2-mu-1}
&\sup_{0<r<1}\int_{0}^{2\pi}\int_{\mathbb D}|f(\zeta_1,re^{i\theta})|^2d\mu_1(\zeta_1)\frac{d\theta}{2\pi} \notag \\
&=\sup_{0<r< 1}\int_{0}^{2\pi}\int_{\mathbb D}|f(\zeta_1,re^{i\theta})-f(0,re^{i\theta})+f(0,re^{i\theta})|^2d\mu_1(\zeta_1)\frac{d\theta}{2\pi} \notag \\
&\leq  2\sup_{0<r< 1}\int_{0}^{2\pi}\int_{\mathbb D}|f(\zeta_1,re^{i\theta})-f(0,re^{i\theta})|^2d\mu_1(\zeta_1)\frac{d\theta}{2\pi} \notag \\ 
& + 2\sup_{0<r< 1}\int_{0}^{2\pi}\int_{\mathbb D}|f(0,re^{i\theta})|^2d\mu_1(\zeta_1)\frac{d\theta}{2\pi} \notag \\
&\overset{\eqref{h-zeta-estimate}}\leq 2 \sup_{0<r< 1}\int_{0}^{2\pi}\int_{\mathbb D}\left \| \frac{f(\cdot,re^{i\theta})-f(\zeta_1,re^{i\theta})}{z_1-\zeta_1}\right \|^2_{H^2(\mathbb D)}d\mu_1(\zeta_1)\frac{d\theta}{2\pi} \notag \\ 
& + 2 \mu_1(\mathbb D) \sup_{0<r< 1}\int_{0}^{2\pi}|f(0,re^{i\theta})|^2\frac{d\theta}{2\pi}. 
\end{align}
Let us assume that $f(z_1,z_2)=\sum_{m,n\geq 0} a_{m,n}z_1^m z_2^n.$ As $f\in H^2(\mathbb D^2)$ by using dominated convergence theorem (see \cite[p.~88]{Ro1988})
\beq\label{Hardy-estimate}
\sup_{0<r< 1}\int_{0}^{2\pi}|f(0,re^{i\theta})|^2\frac{d\theta}{2\pi}
 = \sum_{n\geq 0} |a_{0,n}|^2 \leq  \|f\|^2_{ H^2(\mathbb D^2)}.
\eeq
Combining \eqref{f-in-L2-mu-1} with \eqref{I-mu-formula} and \eqref{Hardy-estimate} yields \eqref{require-int}.

Hence we conclude that $z_1f \in \mathcal D(\mu_1,\mu_2).$ Similarly, one can show show that $z_2f \in \mathcal D(\mu_1,\mu_2).$ Since $\mathcal D(\mu_1,\mu_2)$ is a reproducing kernel Hilbert space so using the closed graph theorem we conclude the result.
\end{proof}

Let $\kappa:\mathbb D\times \mathbb D\to \mathbb C$ denote the reproducing kernel of $\mathcal D(\mu_1,\mu_2).$
\begin{corollary}\label{joint-kernel}
For any $w\in \mathbb D^2,$ $\ker(\mathscr M_z-w) = \{0\}$ and $\ker(\mathscr M^*_z-w)$ is the one-dimensional space spanned by $\kappa(\cdot,\overline{w}).$
\end{corollary}
\begin{proof}
The proof goes the same as that of \cite[Corollary~3.9]{BCG2024}.
\end{proof}

The following lemma recovers a counterpart of \cite[Lemma~2.1]{R1987} for $\mathcal D(\mu).$ 
\begin{lemma}\label{gleason-d-mu}
For any $\mu \in M_+(\overline{\mathbb D}),$ $\mathcal D(\mu)$ has the Gleason property. 
\end{lemma}
\begin{proof}
  Here $\mathscr M_z$ is cyclic on $\mathcal D(\mu)$ then for each $\lambda \in \mathbb D$ the $\dim \ker(\mathscr M^*_z-\overline{\lambda})$ is at most one ( see \cite[Proposition~1.1]{AW1990}). Let $k$ denote the reproducing kernel of $\mathcal D(\mu).$ Then $k(\cdot,\lambda) \in \ker(\mathscr M^*_z-\overline{\lambda})$ so $\ker(\mathscr M^*_z-\overline{\lambda})$ is spanned by $k(\cdot,\lambda).$ For any $h\in \mathcal D(\mu)$ by the reproducing property we know that $h-h(\lambda)$ is orthogonal to $k(\cdot,\lambda).$ That means  $h-h(\lambda)$ belongs the range closure of $(\mathscr M_z-\lambda).$ Since $\mathscr M_z$ is $2$-concave, by \cite[Lemma~1(a)]{R1988} $\mathscr M_z$ is expansive on $\mathcal D(\mu)$ so $(\mathscr M_z-\lambda)$ is bounded below and hence range of $(\mathscr M_z-\lambda)$ is closed. Thus there exists $g\in \mathcal D(\mu)$ such that $h(z)-h(\lambda)=(z-\lambda)g(z)$ for $z\in \mathbb D.$
\end{proof}
The next proposition shows that $\mathcal D(\mu_1,\mu_2)$ has the division property.
\begin{proposition}\label{division-prop}
Let $\mu_1,\mu_2\in M_+(\overline{\mathbb D}).$ Then $\mathcal D(\mu_1,\mu_2)$ has the division property.
\end{proposition}
\begin{proof}
Let $f\in \mathcal D(\mu_1,\mu_2)$ and $\lambda \in \mathbb D$ such that $f(\lambda,z_2)=0$ for each $z_2 \in \mathbb D.$ We are required to show $\frac{f}{z_1-\lambda} $ belongs $\mathcal D(\mu_1,\mu_2).$ Since $H^2(\mathbb D^2)$ has the division property (see \cite[Lemma~4.1]{BCG2024}), there exists $g\in H^2(\mathbb D^2)$ such that $f(z_1,z_2)=(z_1-\lambda)g(z_1,z_2)$ for $z_1,z_2\in \mathbb D.$ Now it boils down to show $g \in \mathcal D(\mu_1,\mu_2).$ From \eqref{f-r-theta-in-d-mu} it is clear that $(z_1-\lambda)g(\cdot,re^{i\theta})\in \mathcal D(\mu_1)$ and $(re^{i\theta}-\lambda)g(re^{i\theta},\cdot)\in \mathcal D(\mu_2)$ for every $r\in(0,1)$ and almost every $\theta \in [0,2\pi].$ Clearly, for every $r\in(0,1)$ and almost every $\theta \in [0,2\pi],$ (by Lemma~\ref{gleason-d-mu}) $g(\cdot,re^{i\theta})\in \mathcal D(\mu_1)$ and  $g(re^{i\theta},\cdot)\in \mathcal D(\mu_2).$ The multiplication operator $\mathscr M_w$ by the coordinate function $w$ is expansive on $\mathcal D(\mu_j),$ $j=1,2.$ So
\beqn
\|g(\cdot,re^{i\theta})\|_{\mathcal D(\mu_1)}\leq \|wg(\cdot,re^{i\theta})\|_{\mathcal D(\mu_1)}.
\eeqn
The rest of the proof is similar to \cite[Proof of Theorem~2.2]{BCG2024}. For the sake of completeness, we are providing the full argument. Note that 
\beqn
\|(w-\lambda)g(\cdot,re^{i\theta})\|_{\mathcal D(\mu_1)} &\geq& (1-|\lambda|)^2 \|g(\cdot,re^{i\theta})\|_{\mathcal D(\mu_1)}\\ &\geq & (1-|\lambda|)^2 \mathcal D_{\mu_1}(g(\cdot,re^{i\theta})).
\eeqn
Integrating both sides with respect to $\theta$
\beqn
&&(1-|\lambda|)^2\int_0^{2\pi}\mathcal D_{\mu_1}(g(\cdot,re^{i\theta})) \frac{d\theta}{2\pi}\\
&&\leq \int_0^{2\pi}\|(w-\lambda)g(\cdot,re^{i\theta})\|^2_{H^2(\mathbb D)}\frac{d\theta}{2\pi}+ \int_0^{2\pi}\mathcal D_{\mu_1}((w-\lambda)g(\cdot,re^{i\theta}))\frac{d\theta}{2\pi}.
\eeqn
By using \cite[Lemma~3.2]{BCG2024} and \eqref{Dirichlet-int-new} we get that
\begin{align*}
 (1-|\lambda|)^2\int_0^{2\pi}\mathcal D_{\mu_1}(g(\cdot,re^{i\theta})) \frac{d\theta}{2\pi}\leq \|(z_1-\lambda)g\|^2_{H^2(\mathbb D^2)} + \mathcal D_{\mu_1,\mu_2}((z_1-\lambda)g).
\end{align*}
Taking supremum over $0<r<1$ on the above inequality gives
\beq\label{g-in-d-mu-1}
\sup_{0<r<1}\int_0^{2\pi}\mathcal D_{\mu_1}(g(\cdot,re^{i\theta})) \frac{d\theta}{2\pi} <\infty.
\eeq
We already have $$\sup_{0<r<1}\int_0^{2\pi} \mathcal D_{\mu_2}((re^{i\theta}-\lambda)g(re^{i\theta},\cdot)) \frac{d\theta}{2\pi} <\infty.$$
So for any $s\in(|\lambda|,1),$
\beqn
&&\sup_{0<r<1}\int_0^{2\pi} \mathcal D_{\mu_2}((re^{i\theta}-\lambda)g(re^{i\theta},\cdot)) \frac{d\theta}{2\pi} \\
&\geq & \int_0^{2\pi} \mathcal D_{\mu_2}((se^{i\theta}-\lambda)g(se^{i\theta},\cdot)) \frac{d\theta}{2\pi} \\
&\geq & (s-|\lambda|)^2  \int_0^{2\pi}  \mathcal D_{\mu_2}(g(se^{i\theta},\cdot)) \frac{d\theta}{2\pi}.
\eeqn
Now taking limit $s\to 1$ gives 
\beq\label{g-in-d-mu-2}
\lim_{s\to 1} \int_0^{2\pi}  \mathcal D_{\mu_2}(g(se^{i\theta},\cdot)) \frac{d\theta}{2\pi} <\infty.
\eeq
Since \cite[Lemma~1.1]{BCG2024} suggest that we can replace the limit by supremum so combining \eqref{g-in-d-mu-1} and \eqref{g-in-d-mu-2} yields $g\in \mathcal D(\mu_1,\mu_2).$

Similarly one can start with the assumption that $f(z_1,\lambda)=0$ for all $z_1 \in \mathbb D$ and show that $\frac{f}{z_2-\lambda} \in \mathcal D(\mu_1,\mu_2).$
\end{proof}
As an application of the Proposition~\ref{division-prop}, we have the following:
\begin{lemma}\label{gleason-proposition}
Gleason problem for $\mathcal D(\mu_1,\mu_2)$ has solution over $\{(\lambda_1,\lambda_2)\in \mathbb D^2: \lambda_1\lambda_2=0\}.$
\end{lemma}
\begin{proof}
Let $f\in \mathcal D(\mu_1,\mu_2)$ and  $\lambda \in \mathbb D.$ It is clear from the Definition~\ref{definition-diri} that 
\beqn
\mathcal D_{\mu_1,\mu_2}(f(\cdot,0))=\mathcal D_{\mu_1}(f(\cdot,0))\leq \mathcal D_{\mu_1,\mu_2}(f), \\
\mathcal D_{\mu_1,\mu_2}(f(0,\cdot))=\mathcal D_{\mu_2}(f(0,\cdot))\leq \mathcal D_{\mu_1,\mu_2}(f).
\eeqn
 Consider the function $h(z_1,z_2)=f(z_1,z_2)-f(z_1,0),$ $(z_1,z_2)\in \mathbb D^2.$ Then $h\in \mathcal D(\mu_1,\mu_2).$ By Proposition~\ref{division-prop}, there exists $f_1\in \mathcal D(\mu_1,\mu_2)$ such that  
 \beq\label{first-part}
  h(z_1,z_2)=f(z_1,z_2)-f(z_1,0)=(z_2-0)f_1(z_1,z_2), \quad (z_1,z_2)\in \mathbb D^2.
 \eeq
 Since $\mathcal D_{\mu_1}(f(\cdot,0))<\infty$ so $f(\cdot,0) \in \mathcal D(\mu_1).$  Applying Lemma~\ref{gleason-d-mu} to  $ \mathcal D(\mu_1)$ we get that for each $\lambda \in \mathbb D,$  there exists $v \in \mathcal D(\mu_1)$ such that 
\beq \label{second-part} 
 f(z_1,0)-f(\lambda,0)=(z_1-\lambda)v(z_1), \quad z_1 \in \mathbb D.
 \eeq
Now adding \eqref{first-part} and \eqref{second-part} gives us
 \beqn
 f(z_1,z_2)-f(\lambda,0)=(z_1-\lambda)v(z_1) + z_2 f_1(z_1,z_2),\quad z_1,z_2\in \mathbb D.
 \eeqn
 Defining $f_2(z_1,z_2)=v(z_1),$  $z_1,z_2\in \mathbb D$ shows that $f_2 \in \mathcal D(\mu_1,\mu_2).$ Thus the Gleason problem has solution at $(\lambda,0)$ for every $\lambda \in \mathbb D.$
 
 Similarly, starting with $H(z_1,z_2)=f(z_1,z_2)-f(0,z_2)$ for $z_1,z_2\in \mathbb D,$ one can show that the Gleason problem can be solved on $\{(0,\lambda): \lambda \in \mathbb D\}.$
\end{proof}
\begin{theorem}
Let $\mu_1, \mu_2 \in M_+(\overline{\mathbb D}).$ Then the followings hold:
\begin{enumerate}[label=$(\alph*)$]
\item  the commuting pair $\mathscr M_z=(\mathscr M_{z_1},\mathscr M_{z_2})$ is cyclic with cyclic vector $1,$
\item Gleason problem can be solved for $\mathcal D(\mu_1,\mu_2)$ on an open neighborhood of $\{(\lambda_1,\lambda_2)\in \mathbb D^2: \lambda_1\lambda_2=0\}.$
\end{enumerate}
\end{theorem}

\begin{proof} $(a)$ Combining Lemmas~\ref{boundedness of M_z},~\ref{polynomials-are-dense} yields the result. 

$(b)$ Lemma~\ref{gleason-proposition} suggests that Gleason problem has solution over $A=\{(\lambda_1,\lambda_2)\in \mathbb D^2: \lambda_1\lambda_2=0\}.$ So for each $\lambda=(\lambda_1,\lambda_2)\in A$ the row operator $T_\lambda:=[\mathscr M_{z_1}-\lambda_1 \; \mathscr M_{z_2}-\lambda_2]$ has closed range. With the help of Corollary~\ref{joint-kernel}
\beqn
\dim \big(\mathcal D(\mu_1,\mu_2)/T_\lambda (\mathcal D(\mu_1,\mu_2)\oplus \mathcal D(\mu_1,\mu_2))\big)&=&\dim \ker T^*_\lambda \\
&=&\dim \ker(\mathscr M^*_z-\overline{\lambda})=1.
\eeqn
Now using \cite[Lemma~4.4]{BCG2024} and the fact that the joint kernel $\ker (\mathscr M_z-\lambda)=\{0\},$ we conclude that the pair $\mathscr M_z-\lambda$ is Fredholm. Thus $A$ is in the complement of the essential spectrum $\sigma_e(\mathscr M_z)$ of $\mathscr M_z.$ Since $\sigma_e(\mathscr M_z)$ is closed, there exists an open subset $V$ of $\mathbb D^2\setminus \sigma_e(\mathscr M_z)$ containing $A.$ Applying \cite[Lemma~5.1]{BCG2024} completes the proof.
\end{proof}

\section{A Representation Theorem}\label{representaion-theorem}
Let $\mu$ be a finite positive Borel measure on $\overline{\mathbb D}.$ For two non-negative integers $i$ and $j,$ the {\it $(i,j)$-th moment of $\mu$} (see \cite{ A1975, D1951, H1935, M1988}) is defined as $$\hat{\mu}\{i,j\}=\int_{\overline{\mathbb D}}(\overline\zeta)^{i} \zeta^{j} d\mu(\zeta).$$
\begin{proposition}
Let $i$ and $j$ be two nonnegative integers and $i\leq j.$ Then  
\beqn
\inp{z^i}{z^j}_{\mathcal D(\mu)}=\delta(i,j)+\sum_{k=0}^{i-1} \hat{\mu}\{j-k-1,i-k-1\},
\eeqn
where $\delta(\cdot,\cdot)$ denotes the two variable Kronecker delta function.
\end{proposition} 
\begin{proof}
Substituting \eqref{d-zeta-int} in \eqref{norm-d-mu} gives
$$\|g\|^2_{\mathcal D(\mu)}=\|g\|^2_{H^2(\mathbb D)}+\int_{\overline{\mathbb D}} \mathcal D_\zeta(g)d\mu(\zeta), \, g\in \mathcal D(\mu).$$
Using polarization identity in the above equation gives 
\begin{align*}
 \inp{z^i}{z^j}_{\mathcal D(\mu)}&=  \inp{z^i}{z^j}_{H^2(\mathbb D)}+ \int_{\overline{\mathbb D}} \left \langle \frac{z^i-\zeta^i}{z-\zeta}, \frac{z^j-\zeta^j}{z-\zeta} \right \rangle_{H^2(\mathbb D)} d\mu(\zeta)  \\
&= \delta(i,j)+ \int_{\overline{\mathbb D}}  \sum_{k=0}^{i-1} \zeta^{i-1-k}(\overline\zeta)^{j-1-k} d\mu(\zeta)\notag \\
&= \delta(i,j)+\sum_{k=0}^{i-1} \int_{\overline{\mathbb D}} \zeta^{i-1-k}(\overline\zeta)^{j-1-k} d\mu(\zeta) \notag\\
&= \delta(i,j)+ \sum_{k=0}^{i-1} \hat{\mu}\{j-k-1,i-k-1\}. \notag
\end{align*}
Hence the result.
\end{proof}
Next, we derive a formula for the inner product of monomials in $\mathcal D(\mu_1,\mu_2).$
 \begin{proposition}\label{inner-product-formula}
Let $\mu_1,\mu_2 \in M_+(\overline {\mathbb D})$ and $m,n,p,q\in \mathbb N$ then 
\beqn
\inp{z^{m}_1z^n_2}{z^{p}_1z^q_2}_{\mathcal D(\mu_1, \mu_2)} = \begin{cases}
0 & \mbox{if}~ m\neq p, ~n \neq q, \\
\inp{z^n_2}{z^q_2}_{\mathcal D(\mu_2)} & \mbox{if}~ m=p, ~n \neq q, \\
\inp{z^{m}_1}{z^{p}_1}_{\mathcal D(\mu_1)} & \mbox{if}~ m \neq p, ~n = q, \\
\|z_1^m\|^2_{\mathcal D(\mu_1)} +\|z_2^n\|^2_{\mathcal D(\mu_2)}-1 & \mbox{if}~ m= p, ~n = q.
\end{cases}
\eeqn
\end{proposition}
\begin{proof}
Using the polarization identity on \eqref{dirichlet-int} gives
\begin{align*}
&\inp{z^{m}_1z^n_2}{z^{p}_1z^q_2}_{\mathcal D(\mu_1, \mu_2)}=\delta(m,p)\delta(n,q)\\
&+  \lim_{r\to 1} \int_0^{2\pi} \int_{\mathbb D} r^{n+q}e^{i(n-q)\theta}mpz_1^{m-1}(\overline z_1)^{p-1} U_{\mu_1}(z_1)dA(z_1) \frac{d\theta}{2\pi}\\
&+  \lim_{r\to 1} \int_0^{2\pi} \int_{\mathbb D} r^{m+p}e^{i(m-p)\theta}nqz_2^{n-1}(\overline z_2)^{q-1} U_{\mu_2}(z_2)dA(z_2) \frac{d\theta}{2\pi}\\
&= \delta(n,q)\delta(m,p)+ \delta(n,q) \int_{\mathbb D} mz_1^{m-1}p(\overline z_1)^{p-1} U_{\mu_1}(z_1)dA(z_1)\\
&+  \delta(m,p) \int_{\mathbb D} nz_2^{n-1}q(\overline z_2)^{q-1} U_{\mu_2}(z_2)dA(z_2).
\end{align*}
Rest follows from \eqref{diri-int-one-variable} and polarisation identity.
\end{proof}
An immediate corollary of the of the above proposition is the following:
\begin{corollary}\label{wandering-s}
For $\mu_1,\mu_2\in M_+(\overline{\mathbb D}),$ the subspace spanned by the constant vector $1$ in  $\mathcal D(\mu_1,\mu_2)$ is a wandering subspace for $\mathscr M_z$ on $\mathcal D(\mu_1,\mu_2).$ 
\end{corollary}
Recall that a commuting pair $T=(T_1,T_2)$ is called a {\it toral $2$-isometry} (see \cite[Eq~(1.1)]{BCG2024}) if it satisfies the equations $I-T^*_iT_i-T^*_jT_j+T^*_iT^*_jT_iT_j=0$ for $i,j=1,2,$ i.e. $\beta_\alpha(T)=0$ for $\alpha\in\{(2,0),(1,1),(0,2)\}.$ For future references we state the following lemma concerning toral $2$-isometry (see \cite[Corollary~3.8]{BCG2024}).
\begin{lemma}\label{toral-cyclic-lemma}
Let the supports of $\mu_1$ and $\mu_2$ be contained in the unit circle. Then the commuting pair $\mathscr M_z$ on $\mathcal D(\mu_1,\mu_2)$ is a cyclic toral $2$-isometry with cyclic vector $1.$
\end{lemma}
Here is a noteworthy observation regarding the commuting pair $\mathscr M_{z}$ on $\mathcal D(\mu_1,\mu_2).$
\begin{lemma}\label{toral-concave}
Let $\mu_1, \mu_2 \in M_+(\overline{\mathbb D}).$ Then  $\mathscr M_z=(\mathscr M_{z_1}, \mathscr M_{z_2})$ is a  toral completely hyperexpansive $2$-tuple with zero defect operator on $\mathcal D(\mu_1,\mu_2).$
\end{lemma}
\begin{proof}
Let $n\geq 2$ and $f\in \mathcal D(\mu_1,\mu_2)$. 
 \begin{align*}
&\inp{\beta_{(n,0)}(\mathscr{M}_{z})(f)}{f}\\
&= \sum_{k=0}^n (-1)^k \binom{n}{k} \|z^k_1f\|^2\\
&= \sum_{k=0}^n (-1)^k \binom{n}{k}\Big( \|z^k_1f\|^2_{H^2(\mathbb D^2)}+ I_{\mu_1,\mu_2}(z^k_1f)+B_{\mu_1,\mu_2}(z^k_1f) \Big) \\
&= \sum_{k=0}^n (-1)^k \binom{n}{k}\Big( \|z^k_1f\|^2_{H^2(\mathbb D^2)}+B_{\mu_1,\mu_2}(z^k_1f) \Big) \\
 &+\sum_{k=0}^n (-1)^k \binom{n}{k} I_{\mu_1,\mu_2}(z^k_1f).
 \end{align*}
By Lemma~\ref{toral-cyclic-lemma} and the fact that every $2$-isometry is automatically a $k$-isometry for each $k \geq 2$ (see \cite{A1990, AS1995}), the first part of the above sum is zero. So we are left with 
\beq\label{beta-n-f}
\inp{\beta_{(n,0)}(\mathscr{M}_{z})(f)}{f}= \sum_{k=0}^n (-1)^k \binom{n}{k} I_{\mu_1,\mu_2}(z^k_1f).
\eeq
Let $k\geq 1.$ Now replacing $f$ by  $z^k_1f$ in \eqref{I-mu-zf} gives us
\begin{align*}
I_{\mu_1,\mu_2}(z^{k+1}_1f)- I_{\mu_1,\mu_2}(z^k_1f)=\lim_{r\to 1}\int_{0}^{2\pi}\int_{\mathbb D} |\zeta_1|^{2k}|f(\zeta_1,re^{i\theta})|^2d\mu_1(\zeta_1)\frac{d\theta}{2\pi}.
\end{align*}
Thus \eqref{beta-n-f} becomes 
\begin{align}\label{beta-n-f-final}
&\inp{\beta_{(n,0)}(\mathscr{M}_{z})(f)}{f} \notag \\
&= \sum_{k=0}^n (-1)^k \binom{n}{k} I_{\mu_1,\mu_2}(z^k_1f) \notag \\
&= \sum_{k=0}^{n-1} (-1)^k \binom{n-1}{k}\left(I_{\mu_1,\mu_2}(z^k_1f)-I_{\mu_1,\mu_2}(z^{k+1}_1f)\right) \notag\\
&=-\lim_{r\to 1}\int_{0}^{2\pi}\int_{\mathbb D} \sum_{k=0}^{n-1} (-1)^k \binom{n-1}{k} |\zeta_1|^{2k}|f(\zeta_1,re^{i\theta})|^2d\mu_1(\zeta_1)\frac{d\theta}{2\pi} \notag\\
&=-\lim_{r\to 1}\int_{0}^{2\pi}\int_{\mathbb D}(1-|\zeta_1|^{2})^{n-1}|f(\zeta_1,re^{i\theta})|^2d\mu_1(\zeta_1)\frac{d\theta}{2\pi}
\leq 0.
\end{align}
Similarly, $\inp{\beta_{(0,n)}(\mathscr{M}_{z})(f)}{f} \leq 0.$ 
We now show that the defect operator $\beta_{(1,1)}(\mathscr M_z)$ of $\mathscr M_{z}$ is zero. If we replace $f$ by  $z_2f$ in \eqref{I-mu-zf},
\begin{align}\label{z1z2-norm}
I_{\mu_1,\mu_2}(z_1z_2f)- I_{\mu_1,\mu_2}(z_2f)&=\lim_{r\to 1}\int_{0}^{2\pi}\int_{\mathbb D}r^2|f(\zeta_1,re^{i\theta})|^2d\mu_1(\zeta_1)\frac{d\theta}{2\pi} \notag\\
&= \lim_{r\to 1}\int_{0}^{2\pi}\int_{\mathbb D}|f(\zeta_1,re^{i\theta})|^2d\mu_1(\zeta_1)\frac{d\theta}{2\pi}.
\end{align}
So applying Lemma~\ref{toral-cyclic-lemma} we have 
\begin{align*}
&\|f\|^2-\|\mathscr{M}_{z_1}f\|^2-\|\mathscr{M}_{z_2}f\|^2+\|\mathscr{M}_{z_1}\mathscr{M}_{z_2}f\|^2 \notag \\
&=\Big( I_{\mu_1,\mu_2}(f) -I_{\mu_1,\mu_2}(z_1f)\Big) +\Big( I_{\mu_1,\mu_2}(z_1z_2f) -I_{\mu_1,\mu_2}(z_2f) \Big) \notag \\
&\overset{\eqref{I-mu-zf} \& \eqref{z1z2-norm}}=0.
\end{align*}
From the discussion after \eqref{beta-alpha} it is clear that $\beta_\alpha(T)\leq 0$ for all $\alpha=(\alpha_1,\alpha_2) \in \mathbb Z_+^2.$ Hence the result. 
\end{proof}

The next lemma is extracted from \cite[Lemma~6.1]{BCG2024} and very useful in the proof of the main theorem.
\begin{lemma}\label{inner-p-simplify}
Let $T=(T_1,T_2)$ is a commuting pair on $\mathcal H$ such that the defect operator (see \eqref{defect-operator}) is zero. Assume that $\ker T^*$ is a wandering subspace of $T.$ Then for each $f_0\in \ker T^*,$ 
\begin{align*}
\inp{T^{m}_1T^n_2f_0}{T^{p}_1T^q_2f_0}_{\mathcal H}= \begin{cases}
0 & \mbox{if}~ m\neq p, ~n \neq q, \\
\inp{T^n_2f_0}{T^q_2f_0}_{\mathcal H} & \mbox{if}~ m=p, ~n \neq q, \\
\inp{T^{m}_1f_0}{T^{p}_1f_0}_{\mathcal H} & \mbox{if}~ m \neq p, ~n = q, \\
\|T^{m}_1f_0\|_{\mathcal H}^2 + \|T^n_2f_0\|_{\mathcal H}^2 - \|f_0\|_{\mathcal H}^2 & \mbox{if}~ m= p, ~n = q.
\end{cases}
\end{align*}  
\end{lemma}
\begin{proof}
From the proof of \cite[Lemma~6.1(i)]{BCG2024} we get that whenever $I-T_1^*T_1-T_2^*T_2+T_1^*T_2^*T_1T_2=0,$ for each $k,l\geq 0,$ 
\beq\label{k,l-induction}
T_1^{*k}T_2^{*l}T_1^kT_2^l=T_1^{*k}T_1^k+T_2^{*l}T_2^l-I.
\eeq
Since $\ker T^*$  is a wandering subspace, using \eqref{k,l-induction} and following the proof of \cite[Lemma~6.1(ii)]{BCG2024} one recovers the required result.
\end{proof}
\begin{proof}[Proof of Theorem~\ref{model-theorem}]
$(B)\implies (A)$ This follows from Corollaries~\ref{joint-kernel}, \ref{wandering-s} and Lemma~\ref{toral-concave}.\\
$(A)\implies (B)$  As $T$ is analytic on $\mathcal H$ so is $T_1$ and $T_2.$ Fix $j\in \{1,2\}.$ Consider the $T_j$ invariant subspace $\mathcal H_j=\overline{span}\{ T^k_jf_0:k\geq 0\}$ of $\mathcal H.$ Then $T_j|_{\mathcal H_j}$ is cyclic analytic completely hyperexpansive operator i.e. an operator of Dirichlet-type (refer to \cite[Definition~1.2, p.70]{A1993}). Hence, by \cite[Theorem~2.5, p.79]{A1993} there exists a unique measure $\mu_j \in M_+(\overline{\mathbb D})$ and a unitary operator $U_j: \mathcal H_j \to \mathcal D(\mu_j)$ such that 
\beq\label{unitary-u-j}
U_jf_0=1, \quad U_jT_j= \mathscr M_w^{(j)} U_j,
\eeq
where $\mathscr M_w^{(j)}$ denotes the multiplication by coordinate function $w$ on $\mathcal D(\mu_j).$ Now consider the map $U$ as 
\beqn
U(T_1^mT_2^nf_0)=z_1^mz_2^n, \quad m,n\geq 0.
\eeqn
 Here we have $\mathcal H=\overline{span}\{T_1^mT_2^nf_0:m,n\geq 0\}$ and $\mathcal D(\mu_1,\mu_2)=\overline{span}\{z_1^m z_2^n: m,n\geq 0\}.$ 
 For any $m,p\geq 0,$ by \eqref{unitary-u-j}
\beqn
\inp{T_j^mf_0}{T_j^pf_0}_{\mathcal H}&=&\inp{U_jT_j^mf_0}{U_jT_j^pf_0}_{\mathcal D(\mu_j)} \\
&=& \inp{(M_w^{(j)})^m U_jf_0}{(M_w^{(j)})^p U_jf_0}_{\mathcal D(\mu_j)} \\
&=& \inp{w^m}{w^p}_{\mathcal D(\mu_j)}. \\
\eeqn 
Now combining Lemma~\ref{inner-p-simplify} and Proposition~\ref{inner-product-formula} yields  $$\inp{T^{m}_1T^n_2f_0}{T^{p}_1T^q_2f_0}_{\mathcal H}=\inp{z_1^mz_2^n}{z_1^pz_2^q}_{\mathcal D(\mu_1,\mu_2)},\; m,n,p,q\geq 0.$$ So $U$ extends as a unitary from $\mathcal H$ onto $\mathcal D(\mu_1,\mu_2).$ Hence the result.
\end{proof}

\begin{corollary}
The commuting pair $\mathscr M_z=(\mathscr M_{z_1}, \mathscr M_{z_2})$ is a toral $2$-isometry if and only if $\mu_1$ and $\mu_2$ are supported on $\partial \mathbb D.$
\end{corollary}
\begin{proof}
As $\beta_{(2,0)}(\mathscr M_z)=0$ putting $n=2$ in \eqref{beta-n-f-final} gives
\begin{align*}
0=\inp{\beta_{(2,0)}(\mathscr{M}_{z})(f)}{f}= -\lim_{r\to 1}\int_{0}^{2\pi}\int_{\mathbb D}(1-|\zeta_1|^2)|f(\zeta_1,re^{i\theta})|^2d\mu_1(\zeta_1)\frac{d\theta}{2\pi}.
\end{align*}
By substituting $f=1$ into the equation above, we find that the support of $\mu_1$ lies outside $\mathbb{D}.$ Similarly, $\beta_{(0,2)}(\mathscr M_z)=0$ implies the support of $\mu_2$ is outside $\mathbb D.$ Conversely, if you assume that $\mu_1$ and $\mu_2$ are supported on the unit circle $\mathbb T$ then by  \cite[Theorem~2.4]{BCG2024} $\mathscr M_z=(\mathscr M_{z_1}, \mathscr M_{z_2})$ becomes a toral $2$-isometry. 
\end{proof}
The following theorem is an extended version of \cite[Theorem~6.4]{BCG2024}(cf. \cite[Theorem~5.2]{R1991}).
 \begin{theorem}
 For $i=1,2$ consider $\mu_1^{(i)},\mu_2^{(i)}\in M_+(\overline{\mathbb D}).$ Then the multiplication $2$-tuple $\mathscr M^{(1)}_z$ on $\mathcal D(\mu_1^{(1)},\mu_2^{(1)})$ is unitarily equivalent to $\mathscr M^{(2)}_z$ on $\mathcal D(\mu_1^{(2)},\mu_2^{(2)})$ if and only if $\mu_j^{(1)}=\mu_j^{(2)},$ $j=1,2.$
 \end{theorem}
 \begin{proof}
Let $\nu_1,\nu_2\in M_+(\overline{\mathbb D})$ and $p$ be a two variable polynomial. We have $\mathcal D_{\nu_1,\nu_2}(z_1p)=I_{\nu_1,\nu_2}(z_1p)+B_{\nu_1,\nu_2}(z_1p)$. 
By \eqref{I-mu-zf},
\beq\label{eq-I}
I_{\nu_1,\nu_2}(z_1p)=\int_{0}^{2\pi}\int_{\mathbb D}|p(\zeta_1,e^{i\theta})|^2d\nu_1(\zeta_1)\frac{d\theta}{2\pi} + I_{\nu_1,\nu_2}(p).
\eeq
From \cite[Lemma~3.5]{BCG2024} we get that 
\beq \label{eq-B}
B_{\nu_1,\nu_2}(z_1p)=\int_{0}^{2\pi}\int_{\mathbb T}|p(\zeta_1,e^{i\theta})|^2d\nu_1(\zeta_1)\frac{d\theta}{2\pi} + B_{\nu_1,\nu_2}(p).
\eeq
Now combing \eqref{eq-I} and \eqref{eq-B} together with \eqref{norm-on-d-mu12} gives
\beq\label{z_1-poly}
\|z_1p\|^2=\|p\|^2+\int_{0}^{2\pi}\int_{\overline{\mathbb D}}|p(\zeta_1,e^{i\theta})|^2d\nu_1(\zeta_1)\frac{d\theta}{2\pi}.
\eeq
One can get a similar expression for $z_2p.$ 
Let $U$ be a unitary map from $\mathcal D(\mu_1^{(1)},\mu_2^{(1)})$ onto $\mathcal D(\mu_1^{(2)},\mu_2^{(2)})$ which satisfies
\beq\label{unitary-eqvl}
U\mathscr M^{(1)}_{z_j}=\mathscr M^{(2)}_{z_j}U,\;j=1,2.
\eeq
Since the joint kernels $\ker \mathscr M^{(1)*}_z$ and $\ker \mathscr M^{(2)*}_z$ are spanned by the constant function $1$ (see Crorllary~\ref{joint-kernel}) so \eqref{unitary-eqvl} suggests that $U^*1\in \ker \mathscr M^{(1)*}_z.$ Hence $U^*1$ must be a unimodular constant. By multiplying suitable unimodular constant one can assume that $U1=1.$ It now follows from \eqref{unitary-eqvl} that $U$ is identity on the polynomials. Thus \eqref{z_1-poly} suggests that for any two-variable polynomial $p,$
\beqn
\int_{0}^{2\pi}\int_{\overline{\mathbb D}}|p(\zeta_1,e^{i\theta})|^2d\mu^{(1)}_1(\zeta_1)\frac{d\theta}{2\pi}
&=& \int_{0}^{2\pi}\int_{\overline{\mathbb D}}|p(\zeta_1,e^{i\theta})|^2d\mu^{(2)}_1(\zeta_1)\frac{d\theta}{2\pi}, \\
\int_{0}^{2\pi}\int_{\overline{\mathbb D}}|p(e^{i\theta},\zeta_2)|^2d\mu^{(1)}_2(\zeta_2)\frac{d\theta}{2\pi}
&=& \int_{0}^{2\pi}\int_{\overline{\mathbb D}}|p(e^{i\theta},\zeta_2)|^2d\mu^{(2)}_2(\zeta_2)\frac{d\theta}{2\pi}.
\eeqn
Thus for any one variable polynomial $p,$
\beqn
\int_{\overline{\mathbb D}} |p(\zeta)|^2 d\mu^{(1)}_j(\zeta) 
=\int_{\overline{\mathbb D}} |p(\zeta)|^2 d\mu^{(2)}_j(\zeta), \quad j=1, 2. 
\eeqn
Using polarization identity and the uniqueness of the two-variable moment problem on $\overline{\mathbb D}$ (see \cite[Remark~1, p.~321]{A1975}) we conclude the theorem.
 \end{proof}

\medskip \textit{Acknowledgement}.  The author extends his gratitude to Prof. Sameer Chavan and Prof. Soumitra Ghara for their valuable suggestions on the subject of this article.

\end{document}